\documentclass[12pt]{article}%
\usepackage{amsmath}
\usepackage{amsfonts}
\usepackage{amssymb}
\usepackage{graphicx}%
\providecommand{\U}[1]{\protect\rule{.1in}{.1in}}
\newtheorem{theorem}{Theorem}

\newtheorem{corollary}[theorem]{Corollary}

\newtheorem{definition}[theorem]{Definition}
\newtheorem{example}[theorem]{Example}

\newtheorem{lemma}[theorem]{Lemma}

\newtheorem{proposition}[theorem]{Proposition}
\newtheorem{remark}[theorem]{Remark}

\newenvironment{proof}[1][Proof]{\noindent\textbf{#1.} }{\ \rule{0.5em}{0.5em}}
\begin{document}

\title{Conditional supremum in Riesz spaces and applications}
\author{Youssef Azouzi\thanks{The authors are members of the GOSAEF research group}
Mohamed Amine Ben Amor,
\and Dorsaf Cherif, Marwa Masmoudi.
\and {\small Research Laboratory of Algebra, Topology, Arithmetic, and Order}\\{\small Department of Mathematics}\\COSAEF {\small Faculty of Mathematical, Physical and Natural Sciences of
Tunis}\\{\small Tunis-El Manar University, 2092-El Manar, Tunisia}}
\maketitle

\begin{abstract}
We extend the concept of conditional supremum to the measure-free setting of
Riesz spaces via the conditional expectation operator. We explore its
properties and show how this tool is crucial in generalizing various results
across multiple disciplines to the framework of Riesz spaces. Among other
applications, we utilize this concept in finance  to derive characterizations
of certain financial conditions.

\end{abstract}

\section{Introduction}

Conditional expectation is a central concept in classical probability theory.
Over the last two decades, Kuo, Labuschagne, and Watson have developed a
theory of stochastic processes in the setting of free-measure theory of Riesz
spaces. Conditional expectation operators have played a fundamental role in
their work. They have focused on studying discrete stochastic processes within
vector lattices. This approach enables the analysis of stochastic processes in
a more general setting than traditional probability theory. After a few years,
Grobler began publishing a series of papers that focused on continuous
processes, some of which were co-authored with Labuschagne.

It turns out that the conditional essential supremum introduced by Barron,
Cardaliaguet and Jensen in \cite{a-1051} is a parallel notion to conditional
expectation and plays a fundamental role in the study of financial processes.
We have observed that conditional supremum can be defined in a different way
that is more suitable to the lattice structure of the underlying space. This
equivalent definition of conditional supremum emphasizes its order properties
and can be extended to the setting of vector lattices. This paper adopts this
idea and continues the development made by the authors mentioned above by
introducing the fundamental notion of conditional supremum in Riesz spaces. We
prove most of the results obtained in \cite{a-1051} by purely algebraic
methods, which helps to uncover the inherent properties of conditional
supremum. The article could serve as a starting point to learn about the
application of Riesz spaces in several areas, notably in finance. It
highlights how Riesz spaces can be used to study the risk associated with
financial assets, as well as analyze the conditions under which arbitrage is
possible. Additionally, Riesz spaces can be useful in portfolio optimization,
which involves selecting a combination of assets to maximize returns while
minimizing risks. By exploring the application of Riesz spaces in finance, one
can gain a deeper understanding of how mathematical concepts can be used to
solve real-world financial problems. For a recent contribution in this
direction with different methods the reader is referred to \cite{L-741}.

The paper is organized as follows. Section 2 provides some preliminaries and
notations about Riesz spaces. Finance vocabulary is introduced only in the
last section, which aims to show that a great development in finance theory
can be made using tools of vector lattices. Section 3 is devoted to the
central notion of this paper, that is, the conditional supremum operator. We
introduce this operator and prove some of its basic properties. We show namely
that this notion is closely connected to the space vector valued norm of the
space $L^{\infty}\left(  \mathbb{F}\right)  $ introduced by Azouzi and
Trabelsi in \cite{L-180}. The final section is devoted to some applications,
in which we show that the sup conditional operator is useful to study
financial markets and ergodic theory in vector lattices. For recent progress
in ergodic theory we refer to \cite{ABMW} and \cite{AM}. In these papers some
fundamental theorems have been generalized to the setting of vector lattices
like Poincar\'{e} Theorem, Kac Formula in and Rokhlin Theorem. For finance we
refer to a recent work by Cherif and Lepinette \cite{CL}, in which they have
employed the concept of conditional supremum in finance to characterize the
weakest form of no arbitrage condition, known as "absence of immediate
profit". This operator also enabled them to construct a novel topology in
$L^{0}\left(  \mathbb{R}\right)  $ that facilitates the definition of
portfolio processes in continuous time as the limit of discrete time processes
with no need for stochastic calculus. Furthermore, they have extended the
theorems of the discrete time case to the continuous time by utilizing this topology.

\section{Preliminaries}

Throughout this section, we provide some background material in Riesz spaces.
A Riesz space is an ordered vector space $(E,\leq)$ where the order is
compatible with the algebraic structure and such that each pair of elements of
$E$ has a supremum and an infimum element. An ideal $I$ in $E$ is a Riesz
subspace such that whenever $f\in I$ and $|g|\leq|f|$ we have $g\in I$. Bands
are ideals that are closed under taking arbitrary supremum. Two elements $a,b$
in $E$ are called disjoint if $0=|a|\wedge|b|:=\inf\left(  \left\vert
a\right\vert ,\left\vert b\right\vert \right)  .$ The band generated by $f\in
E$ is the smallest band containing $f$. Such band is denoted by $B_{f}$ and
called principal. An element $e$ in $E$ is said to be weak order unit if
$E=B_{e}$. Let $D$ be a non-empty subset of $E$. We say that $D$ is upward
directed (denoted $D\uparrow$) if for any two elements $p,q$ in $D$ there
exists $r$ in $D$ such that $r\geq p$ and $r\geq q$. The Riesz space $E$ is
called Dedekind complete whenever every non-empty subset $D$ of $E$ bounded
above has a supremum. If every subset of $E$ consisting of mutually disjoint
elements has a supremum we say that $E$ is laterally complete. A Riesz space
that is both laterally complete and Dedekind complete is called universally
complete Riesz space. The universal completion of $E$, denoted by $E^{u}$, is
the smallest universally complete Riesz space containing $E$.

Throughout this paper, Riesz spaces are assumed to be Dedekind complete and
having weak order units. A triple $\left(  E,e,\mathbb{F}\right)  $ is called
conditional Riesz triple if $E$ is a Dedekind complete Riesz space, $e$ is an
order weak unit and $\mathbb{F}$ is a conditional expectation operator on $E,$
that is, a strictly positive order continuous linear projection with
$\mathbb{F}e=e$ and having a Dedekind complete range.

Recall that, for any conditional expectation $\mathbb{F}$ on $E$, there exists
a largest Riesz subspace of the universal completion $E^{u}$ of $E$, called
the natural domain of $\mathbb{F}$ and denoted by $\mathcal{L}^{1}%
(\mathbb{F})$, to which $\mathbb{F}$ extends uniquely to a conditional
expectation. More information about this can be found in \cite{L-24}. More
general one can define spaces $L^{p}\left(  \mathbb{F}\right)  $ for
$p\in\left(  1,\infty\right)  $ as follows:%

\[
L^{p}(\mathbb{F})=\{x\in L^{1}(\mathbb{F}):|x|^{p}\in L^{1}(\mathbb{F})\},
\]
and we equip this space with a vector valued norm:
\[
N_{p}(x)=\mathbb{F}(|x|^{p})^{1/p}\text{ for all }x\in L^{p}(\mathbb{F}).
\]

On the other hand the space $L^{\infty}\left(  \mathbb{F}\right)  $ is the
ideal of $E^{u}$ generated by $R\left(  \mathbb{F}\right)  $ and it is
equipped with the $R\left(  \mathbb{F}\right)  $-vector valued norm
$\left\Vert .\right\Vert _{\mathbb{F},\infty}$ :%

\begin{align*}
L^{\infty}\left(  \mathbb{F}\right)   &  =\left\{  x\in L^{1}\left(
\mathbb{F}\right)  :\left\vert f\right\vert \leq u\text{ for some }u\in
R\left(  \mathbb{F}\right)  \right\}  .\\
N_{\infty}\left(  x\right)   &  =\inf\left\{  u\in R\left(  \mathbb{F}\right)
:\left\vert x\right\vert \leq u\right\}  \text{ for all }x\in L^{\infty
}\left(  \mathbb{F}\right)  .
\end{align*}

It is worthy to note that the spaces $L^{p}\left(  \mathbb{F}\right)  $ have
the structure of $R\left(  T\right)  $-module. For this reason a suitable
definition of $\mathcal{L}^{\infty}(\mathbb{F})$ is that it is the ideal
generated by $R\left(  \mathbb{F}\right)  $ and not simply the principal ideal
generated by $e.$

A filtration on a Dedekind complete Riesz space $E$ is a family of conditional
expectations $(\mathbb{F}_{i})_{i\in\mathbb{T}}$, on $E$ with $\mathbb{F}%
_{i}\mathbb{F}_{j}=\mathbb{F}_{j}\mathbb{F}_{i}=\mathbb{F}_{i}$ for all $j>i$.
An adapted process is just a sequence $\left(  x_{i}\right)  $ in $E$ such
that $x_{i}\in R\left(  \mathbb{F}_{i}\right)  .$

In some stage of our work we will need some notions about $f$-algebras. We
refer to \cite{b-1087} for this topic and we will just mention a few results
about it in subsection \ref{Pr}.

\section{Conditional Supremum in a Riesz space}

\subsection{Definitions and elementary properties}

We consider in this section a conditional Riesz triple $(E,e,\mathbb{F})$ as
it is defined above. We recall from \cite{L-180} that the space $\mathcal{L}%
^{\infty}(\mathbb{F})$ is equipped with a vector-valued norm given by
\[
||f||_{\infty,\mathbb{F}}=\inf\{g\in R(\mathbb{F}):|f|\leq g\}.
\]

This norm is closely linked to a concept introduced in \cite{a-1051} by
Barron, Cardaliaguet and Jensen called conditional essential supremum. We
recall from \cite{a-1051} that if $X$ is a random variable on a probability
space ($\Omega,\mathcal{F},\mathbb{P}$) and $\mathcal{H}$ is a sub-$\sigma
$-algebra then the conditional essential supremum given $\mathcal{H}$ is the
unique $\mathcal{H}$-measurable random variable $Y$ (up to equivalence)$,$
denoted by $\mathcal{M}(X|\mathcal{H})$, satisfying :%
\[
\operatorname*{esssup}\limits_{\omega\in A}X(w)=\operatorname*{esssup}%
\limits_{\omega\in A}Y\;\;\text{for all}\;A\text{ in}\;\mathcal{H}.
\]
It was shown in \cite{a-1051} that such a random variable exists. Connection
between the notions of conditional supremum and the vector valued norm
$\left\Vert .\right\Vert _{\mathbb{F},\infty}$ with respect to a condition
expectation is not obvious if we use the above definition of $\mathcal{M}%
(X|\mathcal{H})$. A useful characterization of the conditional supremum given
by the authors of \cite{a-1051} asserts that $\mathcal{M}(X|\mathcal{H})$ is
the smallest $\mathcal{H}$-measurable function larger than $X$. This
characterization provides an equivalent definition which is more suitable in
the free measure setting.

Observe that if $\mathbb{F}=\mathbb{E}^{\mathcal{H}},$ the conditional
expectation given the sub-$\sigma$-algebra $\mathcal{H}.$ then the range of
$\mathbb{F}$ is $\mathcal{L}^{1}(\mathcal{H})$ and we have the following
formula%
\[
||X||_{\infty,\mathbb{E}^{\mathcal{H}}}=\mathcal{M}(|X||\mathcal{H}),
\]
which makes a link between the definition of conditional essential supremum
introduced in \cite{a-1051} and that vector-valued norm of $\mathcal{L}%
^{\infty}\left(  \mathbb{F}\right)  $ introduced in \cite{L-180}. This
motivates the following definition.

\begin{definition}
\label{def3}Let $(E,e,\mathbb{F})$ be a conditional Riesz triple and
$f\in\mathcal{L}^{\infty}(\mathbb{F})$. The conditional supremum of $f$ with
respect to $\mathbb{F}$ is given by
\[
M_{\mathbb{F}}(f)=\inf{\{g\in R(\mathbb{F}),f\leq g}\}.
\]

\end{definition}

The fact that $f$ belongs to $\mathcal{L}^{\infty}(\mathbb{F})$ ensures the
existence of this supremum.

It follows that formula (1) can be written%
\[
M_{\mathbb{F}}(|f|)=||f||_{\infty,\mathbb{F}}.
\]

Let us mention a simple but very useful fact which follows immediately from
the above definition and noted here for further references.

\begin{remark}
\label{rq}If $g$ is in $R(\mathbb{F})$ and $f\leq g$ then $M_{\mathbb{F}%
}(f)\leq g$.
\end{remark}

In a similar way, we define the conditional infimum of $f$ with respect to
$\mathbb{F}$ in $\mathcal{L}^{\infty}(\mathbb{F})$ as follows%
\[
m_{\mathbb{F}}(f)=\sup{\ \{g\in R(\mathbb{F}),g\leq f}\}.
\]

The maps $M_{\mathbb{F}}$ and $m_{\mathbb{F}}$ are connected by the relation
$m_{\mathbb{F}}(-f)=-M_{\mathbb{F}}(f)$. It is also clear that%
\[
m_{\mathbb{F}}(f)\leq f\leq M_{\mathbb{F}}(f),
\]
with equalities whenever $f$ is in the range of $\mathbb{F}$.

The next result collects some basic properties of the conditional supremum and
the conditional infimum with respect to a conditional expectation operator.

\begin{proposition}
\label{4B}Let $(E,e,\mathbb{F})$ be a conditional Riesz triple. Then the
following statements hold.

\begin{enumerate}
\item[(i)] $M_{\mathbb{F}}$ and $m_{\mathbb{F}}$ are idempotent and increasing
operators from $L^{\infty}\left(  \mathbb{F}\right)  $ to $R\left(
\mathbb{F}\right)  .$

\item[(ii)] The map $M_{\mathbb{F}}$ is positively homogeneous and
sub-additive. That is,%
\[
M_{\mathbb{F}}(\lambda f)=\lambda M_{\mathbb{F}}(f),\qquad\text{and}\qquad
M_{\mathbb{F}}(f+g)\leq M_{\mathbb{F}}(f)+M_{\mathbb{F}}(g).
\]
for all$\;\lambda\in\mathbf{R}_{+}$ and $f,g\in\mathcal{L}^{\infty}\left(
\mathbb{F}\right)  .$\newline If, in addition, $g\in R(\mathbb{F})$, then we
have the equalities%
\[
M_{\mathbb{F}}(f+g)=g+M_{\mathbb{F}}(f)\text{ and }m_{\mathbb{F}}\left(
f+g\right)  =g+m_{\mathbb{F}}\left(  f\right)  .
\]

\item[(iii)] We have for all $f,g\in L^{\infty}\left(  \mathbb{F}\right)  ,$%
\[
\left\vert M_{\mathbb{F}}(f)-M_{\mathbb{F}}(g)\right\vert \leq M_{\mathbb{F}%
}(|f-g|),\qquad\text{and}\qquad|m_{\mathbb{F}}(f)-m_{\mathbb{F}}(g)|\leq
M_{\mathbb{F}}(|f-g|).
\]
In particular, $\left\vert M_{\mathbb{F}}(f)\right\vert \leq M_{\mathbb{F}%
}(\left\vert f\right\vert ).$
\end{enumerate}
\end{proposition}

\begin{proof}
The proof of (i) is obvious. For (ii) we only need to prove the last part.
Assume then that $f\in\mathcal{L}^{\infty}\left(  \mathbb{F}\right)  $ and
$g\in R\left(  \mathbb{F}\right)  .$ The sub-additivity of $M_{\mathbb{F}}$
yields%
\[
M_{\mathbb{F}}(f+g)\leq M_{\mathbb{F}}(f)+g.
\]
The reverse inequality can be obtained by applying the first one to $f+g$ and
$-g$,%
\[
M_{\mathbb{F}}(f)\leq M_{\mathbb{F}}(f+g)-g.
\]

(iii) The first inequality is fulfilled as the map $f\rightarrow
M_{\mathbb{F}}\left(  \left\vert f\right\vert \right)  $ is a vector-valued
norm. For the converse observe that%
\[
m_{\mathbb{F}}(f)-g\leq f-g\leq M_{\mathbb{F}}(|f-g|).
\]
We apply then $M_{\mathbb{F}}$ to get%
\[
m_{\mathbb{F}}(f)-m_{\mathbb{F}}(g)=m_{\mathbb{F}}(f)+M_{\mathbb{F}}(-g)\leq
M_{\mathbb{F}}(|f-g|),
\]
and we derive finally that%
\[
|m_{\mathbb{F}}(f)-m_{\mathbb{F}}(g)|\leq M_{\mathbb{F}}(|f-g|).
\]
This completes the proof.
\end{proof}

As $M_{\mathbb{F}}(f)=-m_{\mathbb{F}}(-f)$, we obtain analogous properties for
conditional infimum. Namely, $m_{\mathbb{F}}$ is super-additive and positively
homogeneous, which gives%
\[
M_{\mathbb{F}}(\lambda f)=\lambda m_{\mathbb{F}}(f)\;\text{for all}%
\;\lambda\in\mathbb{R}_{-}.
\]

The next proposition provides a generalization of \cite[Proposition
2.9]{a-1051}. As noticed in the introduction, our approach helps to get more
simplified and transparent proofs. It fact it is enough to apply $\mathbb{F}$
to the following%
\[
m_{\mathbb{F}}(f)\leq f\leq M_{\mathbb{F}}(f).
\]

\begin{proposition}
\label{4C}Let $(E,e,\mathbb{F})$ be a conditional Riesz triple. Then for every
$f$ in $\mathcal{L}^{\infty}(\mathbb{F})$, we have%
\[
m_{\mathbb{F}}(f)\leq\mathbb{F}(f)\leq M_{\mathbb{F}}(f).
\]

\end{proposition}

A very important and useful property of conditional expectation is the tower
property; see \cite{b-23}. An analogous tower property for conditional
supremum has been obtained in \cite[Theorem 2.8]{a-1051}. Its version in Riesz
spaces setting will be our next result.

\begin{proposition}
Let $(E,e,\mathbb{F})$ be a conditional Riesz triple and let $\mathbb{G}$ be
another conditional expectation on $E$ satisfying $\mathbb{GF}=\mathbb{FG}%
=\mathbb{F}$ and let $f\in\mathcal{L}^{\infty}(\mathbb{F}).$ Then the
following hold.

\begin{enumerate}
\item $M_{\mathbb{G}}(f)\leq M_{\mathbb{F}}(f);$

\item $M_{\mathbb{F}}(M_{\mathbb{G}}(f))=M_{\mathbb{G}}(M_{\mathbb{F}%
}(f))=M_{\mathbb{F}}(f).$
\end{enumerate}
\end{proposition}

\begin{proof}
The equality $M_{\mathbb{G}}(M_{\mathbb{F}}(f))=M_{\mathbb{F}}(f)$ is clearly
true because
\[
M_{\mathbb{F}}(f)\in R(\mathbb{F})\subset R(\mathbb{G}).
\]
Also as $M_{\mathbb{F}}(f)$ belongs to $R(\mathbb{G})$ and is an upper bound
of $f$ we have
\[
M_{\mathbb{G}}(f)\leq M_{\mathbb{F}}(f).
\]
From this we deduce that%
\[
M_{\mathbb{F}}(M_{\mathbb{G}}(f))\leq M_{\mathbb{F}}(f).
\]
The converse inequality holds as $M_{\mathbb{F}}$ is increasing and $f\leq
M_{\mathbb{G}}(f).$
\end{proof}

The following result could be compared with the classical case in
\cite[Proposition 2.1]{a-1051}, on which we give the behavior of operators
$m_{\mathbb{F}}$ and $M_{\mathbb{F}}$ with lattice operations.

\begin{proposition}
\label{4D}Let $(E,e,\mathbb{F})$ be a conditional Riesz triple. Then the
following statements hold for all $f,g\in\mathcal{L}^{\infty}(\mathbb{F})$.

\begin{enumerate}
\item[(i)] $M_{\mathbb{F}}(f\vee g)=M_{\mathbb{F}}(f)\vee M_{\mathbb{F}}(g)$
and $m_{\mathbb{F}}(f\wedge g)=m_{\mathbb{F}}(f)\wedge m_{\mathbb{F}}(g).$

\item[(ii)] $M_{\mathbb{F}}(g\wedge f)\leq M_{\mathbb{F}}(g)\wedge
M_{\mathbb{F}}(f)$ and $m_{\mathbb{F}}(f\vee g)\geq m_{\mathbb{F}}(f)\vee
m_{\mathbb{F}}(g).$\newline If, in addition, $g\in R(\mathbb{F})$, these
inequalities become equalities.

\item[(iii)] $m_{\mathbb{F}}(f\wedge g)\leq m_{\mathbb{F}}(f)\wedge
M_{\mathbb{F}}(g)\leq M_{\mathbb{F}}(f\wedge g)$ and $m_{\mathbb{F}}(f\vee
g)\leq m_{\mathbb{F}}(f)\vee M_{\mathbb{F}}(g)\leq M_{\mathbb{F}}(f\vee g).$

\item[(iv)] if $f\wedge g=0$, then $m_{\mathbb{F}}(f)\wedge M_{\mathbb{F}%
}(g)=0$.
\end{enumerate}
\end{proposition}

\begin{proof}
In (i), (ii) we will only prove the statement for $M_{\mathbb{F}},$ the
analogous result for $m_{\mathbb{F}}$ can be deduced easily$.$

(i) The inequality%
\[
M_{\mathbb{F}}(f\vee g)\leq M_{\mathbb{F}}(f)\vee M_{\mathbb{F}}(g)
\]
is satisfied because $f\vee g\leq M_{\mathbb{F}}(f)\vee M_{\mathbb{F}}(g)\in
R\left(  \mathbb{F}\right)  .$ The converse inequality holds as $M_{\mathbb{F}%
}$ is increasing and then%
\[
M_{\mathbb{F}}(f)\vee M_{\mathbb{F}}(g)\leq M_{\mathbb{F}}(f\vee g).
\]

(ii) The inequality$M_{\mathbb{F}}(g\wedge f)\leq M_{\mathbb{F}}(g)\wedge
M_{\mathbb{F}}(f)$ is obvious. Now pick $g$ in $R(\mathbb{F})$ and suppose, by
the way of contradiction, that $M_{\mathbb{F}}(g\wedge f)<g\wedge
M_{\mathbb{F}}(f)$. Then, there exists $h\in R(\mathbb{F})$ such that%
\[
g\wedge f<h<g\wedge M_{\mathbb{F}}(f),
\]
which implies that%
\[
0<g\wedge M_{\mathbb{F}}(f)-h\leq g\wedge M_{\mathbb{F}}(f)-g\wedge f\leq
M_{\mathbb{F}}(f)-f.
\]
This contradicts the definition of $M_{\mathbb{F}}\left(  f\right)  $ and
completes the proof of (ii).

(iii) Using (ii) and the fact that $m_{\mathbb{F}}$ and $M_{\mathbb{F}}$ are
increasing we have%
\begin{align*}
m_{\mathbb{F}}(f\wedge g)  &  \leq m_{\mathbb{F}}\left(  f\wedge
M_{\mathbb{F}}\left(  g\right)  \right)  =m_{\mathbb{F}}\left(  f\right)
\wedge M_{\mathbb{F}}\left(  g\right) \\
&  \leq f\wedge M_{\mathbb{F}}\left(  g\right)  \leq M_{\mathbb{F}}\left(
f\wedge M_{\mathbb{F}}\left(  g\right)  \right)  =M_{\mathbb{F}}\left(
f\right)  \wedge M_{\mathbb{F}}\left(  g\right)  .
\end{align*}
For the statement with $\vee$ we need also the inequality $m_{\mathbb{F}}\leq
M_{\mathbb{F}}.$ We get then%
\begin{align*}
m_{\mathbb{F}}(f\vee g)  &  \leq m_{\mathbb{F}}(f\vee M_{\mathbb{F}}\left(
g\right)  )=m_{\mathbb{F}}(f)\vee M_{\mathbb{F}}\left(  g\right) \\
&  \leq M_{\mathbb{F}}(f)\vee M_{\mathbb{F}}\left(  g\right)  =M_{\mathbb{F}%
}(f\vee g).
\end{align*}

(iv) follows immediately from (iii).1
\end{proof}

\textbf{Note. }In general, inequalities (ii) can be strict. Take for example
$E=L^{1}\left(  \left[  0,1\right]  ,\Sigma,\mathcal{\mu}\right)  ,$ where
$\Sigma$ is the Lebesgue $\sigma$-algebra and $\mu$ is the Lebesgue measure.
Let $\mathbb{F}=\mathbb{E}$ be the expectation operator. In this case
$M_{\mathbb{F}}\left(  h\right)  =\left\Vert h\right\Vert _{\infty}$ for $h\in
E^{+}.$ So for $f=\chi_{\left[  0;1/2\right]  }$ and $g=\chi_{(1/2,1]}$ we
have $0=M_{\mathbb{F}}(g\wedge f)<M_{\mathbb{F}}(g)\wedge M_{\mathbb{F}%
}(f)=1.$

The next result is an immediate consequence of the precedent proposition.

\begin{corollary}
Let $(E,e,\mathbb{F})$ be a conditional Riesz triple. Every $f\in
\mathcal{L}^{\infty}(\mathbb{F})$ satisfies

\begin{enumerate}
\item $m_{\mathbb{F}}(f^{+})=m_{\mathbb{F}}(f)^{+}$ and $m_{\mathbb{F}}%
(f^{-})=M_{\mathbb{F}}(f)^{-}.$

\item $M_{\mathbb{F}}(f^{+})=M_{\mathbb{F}}(f)^{+}$ and $M_{\mathbb{F}}%
(f^{-})=m_{\mathbb{F}}(f)^{-}.$
\end{enumerate}
\end{corollary}

\subsection{Convexity}

Riesz spaces provide an ideal framework for developing an abstract theory of
convexity. In this section, We present few properties and formulas dealing
with convexity.

\begin{proposition}
Let $(E,e,\mathbb{F})$ be a conditional Riesz triple. Let $\varphi
:\mathbb{R}\rightarrow\mathbb{R}$ be a convex function and $f\in
\mathcal{L}^{\infty}(\mathbb{F})$. Then%
\[
\varphi(\mathbb{F}x)\leq\mathbb{F}\varphi(x)\leq M_{\mathbb{F}}(\varphi(x)).
\]

\end{proposition}

\begin{proof}
The first inequality has been shown by Grobler in \cite{L-06}. The second one
follows from Proposition \ref{4C}.
\end{proof}

The following theorem provides a version of Jensen Inequality for conditional
supremum in Riesz spaces.

\begin{theorem}
Let $(E,e,\mathbb{F})$ be a conditional Riesz triple. Let $\varphi
:\mathbb{R}\rightarrow\mathbb{R}$ be a convex function and $f\in
\mathcal{L}^{\infty}(\mathbb{F})$. Then
\[
\varphi(M_{\mathbb{F}}(x))\leq M_{\mathbb{F}}(\varphi(x)).
\]

\end{theorem}

\begin{proof}
(i) Assume first that $\varphi$ is an affine function and write $\varphi
(x)=ax+b.$ If $a\geq0$ then using Proposition \ref{4B}.(ii) we have%
\[
M_{\mathbb{F}}(\varphi(x))=M_{\mathbb{F}}(ax+be)=aM_{\mathbb{F}}%
(x)+be=\varphi(M_{\mathbb{F}}(x)).
\]
If $a\leq0$, then%
\begin{align*}
M_{\mathbb{F}}(\varphi(x))  &  =M_{\mathbb{F}}(ax+b)=-aM_{\mathbb{F}%
}(-x)+b\;e=am_{\mathbb{F}}(x)+be\\
&  \geq aM_{\mathbb{F}}(x)+be=\varphi(M_{\mathbb{F}}(x)).
\end{align*}

(ii) Assume now that $\varphi=\varphi_{1}\vee\varphi_{2}\vee...\vee\varphi
_{n}$ for some affine functions $\varphi_{1},....,\varphi_{n}.$ By Proposition
\ref{4D}.(i) and the first case, we get%
\begin{align*}
M_{\mathbb{F}}(\varphi(x))  &  =(M_{\mathbb{F}}(\varphi_{1}x)\vee
...\vee(M_{\mathbb{F}}(\varphi_{n}x))\\
&  \geq\varphi_{1}(M_{\mathbb{F}}(x))\vee...\vee\varphi_{n}(M_{\mathbb{F}%
}(x))=\varphi(M_{\mathbb{F}}(x)).
\end{align*}

(iii) Finally assume that $\varphi$ is a convex function. Then there exists an
increasing sequence $(\varphi_{p})_{p}$ of functions each one of them is
finite supremum of affine functions such that $\varphi_{p}\uparrow\varphi$
(see \cite[Corollary 4.2]{L-06}). We have by (ii),
\[
\varphi_{p}\left(  M_{\mathbb{F}}\left(  x\right)  \right)  \leq
M_{\mathbb{F}}\left(  \varphi_{p}\left(  x\right)  \right)  \leq
M_{\mathbb{F}}\left(  \varphi\left(  x\right)  \right)  .
\]
Now using \cite[Proposition 3.8]{L-06} we get the expected inequality. This
ends the proof.
\end{proof}

\begin{corollary}
Let $\left(  \Omega,\mathcal{F},\mathbb{P}\right)  $ be a probability triple
and $\mathcal{H}$ a sub-$\sigma$-algebra of $\mathcal{F}.$ Let $X$ be a random
variable in $L^{\infty}\left(  \Omega\right)  $ and $\varphi$ a real convex
function. Then
\[
\varphi\left(  M\left(  X|\mathcal{H}\right)  \right)  \leq M\left(
\varphi\left(  X\right)  |\mathcal{H}\right)  .
\]

\end{corollary}

\subsection{Convergence}

For a sequence of random variables in $\mathcal{L}^{\infty}\left(
\Omega\right)  $ the fact that $X_{n}$ converges almost surely to $X$ does not
imply in general that $X_{n}$ converges to $X$ in $\mathcal{L}^{\infty}$-norm.
So we can not expect to prove that order convergence of nets implies
convergence with the respect to the vector-valued norm $M_{\mathbb{F}}\left(
\left\vert .\right\vert \right)  $. The next is stated for nets and can be
compared with \cite[Proposition 2.1]{a-1051}.

\begin{lemma}
Let $(f_{\alpha})_{\alpha}$ be a net in $\mathcal{L}^{\infty}\left(
\mathbb{F}\right)  .$

\begin{enumerate}
\item If $f_{\alpha}\uparrow f$ in $\mathcal{L}^{\infty}\left(  \mathbb{F}%
\right)  $ then $M_{\mathbb{F}}(f_{\alpha})\uparrow M_{\mathbb{F}}(f)$.

\item If $f_{\alpha}\downarrow f$ in $\mathcal{L}^{\infty}\left(
\mathbb{F}\right)  $ then $m_{\mathbb{F}}(f_{\alpha})\downarrow m_{\mathbb{F}%
}(f)$.
\end{enumerate}
\end{lemma}

\begin{proof}
(i) Obviously $\sup\limits_{\alpha}{(M_{\mathbb{F}}(f_{\alpha}))}\leq
M_{\mathbb{F}}(f)$. On the other hand, we have $f_{\alpha}\leq M_{\mathbb{F}%
}(f_{\alpha})$ and then $f=\sup{(f_{\alpha})}\leq\sup{(M_{\mathbb{F}%
}(f_{\alpha}))}\in R(\mathbb{F}).$ It follows from the definition of
$M_{\mathbb{F}}\left(  f\right)  $ that $M_{\mathbb{F}}(f)\leq\sup
{(M_{\mathbb{F}}(f_{\alpha}))}$.

(ii) It follows from (i).
\end{proof}

If $f_{\alpha}\downarrow f$, then it does not follow that $M_{\mathbb{F}%
}(f_{\alpha})\downarrow M_{\mathbb{F}}(f)$. This fails even for the classical
case when $\mathbb{F}$ is the expectation operator (See \cite[Remark
2.2]{a-1051}).

In \cite{a-1051}, the authors extend the well-known result for limits of
$L^{p}$ when $p\rightarrow\infty$. Actually, in a probability space
($\Omega,\mathcal{F},\mathcal{P}$), if $f$ is a non negative random variable
in $\mathcal{L}^{\infty}(\Omega)$ and $\mathcal{H}$ is a sub-$\sigma$-algebra
of $\mathcal{F}$, then%
\[
\lim_{p\rightarrow\infty}E(f^{p}|\mathcal{H})^{1/p}=\operatorname*{esssup}%
{}_{\mathcal{H}}(f).
\]
The generalization of this result in Riesz spaces framework is due to Azouzi
and Trabelsi in \cite{L-180}. We just reformulate this result in terms of
conditional supremum in the following proposition.

\begin{proposition}
\label{4E}Let $(E,e,\mathbb{F})$ be a conditional Riesz triple. Then%
\[
\lim_{p\rightarrow\infty}\mathbb{F}(f^{p})^{1/p}=M_{\mathbb{F}}(f),\qquad
\text{for all }f\in\mathcal{L}^{\infty}(\mathbb{F})^{+}.
\]

\end{proposition}

\subsection{The product\label{Pr}}

Kuo, Labuschagne and Watson in (\cite{L-24}) have proved the averaging
property for conditional expectations : If $f,g\in E,$ $f\in R\left(
\mathbb{F}\right)  $ and $fg\in E$ then $\mathbb{F}\left(  fg\right)
=f\mathbb{F}\left(  g\right)  .$ Then then Azouzi and Trabelsi remarked in
\cite{L-180} that the condition $fg\in L^{1}\left(  \mathbb{F}\right)  $ is
automatically satisfied if $f\in R\left(  \mathbb{F}\right)  $ and $g\in
E\subseteq L^{1}\left(  \mathbb{F}\right)  $ as $L^{p}\left(  \mathbb{F}%
\right)  $ is an $R\left(  \mathbb{F}\right)  $-module. We will prove a
similar result for conditional supremum. In the classical case, this property
has not been treated by Barron, Cardaliaguet and Jensen in \cite{a-1051}. We
will use some $f$-algebra techniques to obtain a result in the abstract case.
We will recall some basic facts about $f$-algebras that will be needed here.
For more details the reader is referred to \cite{b-1087}. We say that $A$ is
an $\ell$-algebra if $A$ is a Riesz space equipped with a multiplication that
gives $A$ a structure of algebra such that the product of two positive
elements is positive. An $f$-algebra is an $\ell$-algebra satisfying the
following property : for every $a,b,c\in A^{+}$ we have $a\wedge b=0$ implies
$ac\wedge bc=ca\wedge cb=0.$ Any Archimedean $f$-algebra is commutative. If
$E$ is a Dedekind complete Riesz space with weak order unit $e\in E^{+}$ then
its universal completion $E^{u}$ has a structure of $f$-algebra with unit
$e.$\cite[p. 127]{b-240} We will use the following facts :

(i) if $a,b\in A^{+}$ then $ab=\left(  a\vee b\right)  \left(  a\wedge
b\right)  $;

(ii) If $x\geq e$ then $x$ is invertible.

\begin{theorem}
\label{4F}Let $(E,e,\mathbb{F})$ be a conditional Riesz triple. If
$f,h\in\mathcal{L}^{\infty}(\mathbb{F})$ with $h\in R(\mathbb{F})^{+}$, then%
\[
M_{\mathbb{F}}(hf)=hM_{\mathbb{F}}(f).
\]

\end{theorem}

\begin{proof}
(i) Assume first that $0\leq h\leq e$. Observe that $hf\leq hM_{\mathbb{F}%
}(f)\in R\left(  \mathbb{F}\right)  .$ It follows that
\[
hf\leq M_{\mathbb{F}}(hf)\leq hM_{\mathbb{F}}(f).
\]
Similarly, we have%
\[
\left(  e-h\right)  f\leq M_{\mathbb{F}}(\left(  e-h\right)  f)\leq\left(
e-h\right)  M_{\mathbb{F}}(f).
\]
Summing these two previous inequalities we obtain%
\[
f\leq M_{\mathbb{F}}(hf)+M_{\mathbb{F}}((e-h)f)\leq M_{\mathbb{F}}(f),
\]
which implies that%
\[
hM_{\mathbb{F}}(f)+(e-h)M_{\mathbb{F}}(f)=M_{\mathbb{F}}(f)=M_{\mathbb{F}%
}(hf)+M_{\mathbb{F}}((e-h)f),
\]
It follows that the inequalities%
\[
hM_{\mathbb{F}}(f)\geq M_{\mathbb{F}}(hf)\text{ and }(e-h)M_{\mathbb{F}%
}(f)\geq M_{\mathbb{F}}((e-h)f),
\]
are in fact equalities and we conclude that%
\[
M_{\mathbb{F}}(hf)=hM_{\mathbb{F}}(f).
\]

(ii) Assume now that $h\geq e$. We have already seen that $hM_{\mathbb{F}%
}(f)\geq M_{\mathbb{F}}(hf)$ (which is valid for all $h$ in $R\left(
\mathbb{F}\right)  ^{+}$). By \cite[Theorem 146.3]{b-1087} we know that $h$ is
invertible, and that $h^{-1}\leq e.$ Apply now the first case to get%
\[
h^{-1}M_{\mathbb{F}}(hf)=M_{\mathbb{F}}(h^{-1}hf)=M_{\mathbb{F}}(f).
\]
The required equality then follows by multiplying by $h.$

(iii) For the general case, we use the formula
\[
h=he=(h\wedge e)(h\vee e),
\]
mentioned above and can be found in \cite[Theorem 142.4]{b-1087}. According to
(i) and (ii) we have%
\begin{align*}
M_{\mathbb{F}}(hf)  &  =M_{\mathbb{F}}((h\wedge e)(h\vee e)f)=(h\wedge
e)M_{\mathbb{F}}\left(  h\vee e)f\right) \\
&  =(h\wedge e)\left(  h\vee e\right)  M_{\mathbb{F}}\left(  f\right)
=hM_{\mathbb{F}}(f).
\end{align*}
The proof is complete.
\end{proof}

\begin{remark}
In the case when $f\in E^{+}$, we can provide two other proofs for our
theorem. One using the freudenthal theorem and the second using proposition
\ref{4E}. We present the last one as it is somehow shorter and more elegant.
If $f\geq0$ and $h\geq0$ then using the averaging property of $\mathbb{F}$ we
have%
\[
\mathbb{F}\left(  \left(  hf\right)  ^{p}\right)  ^{1/p}=\mathbb{F}\left(
h^{p}f^{p}\right)  ^{1/p}=\left[  h^{p}\mathbb{F}\left(  f^{p}\right)
\right]  ^{1/p}=h\left[  \mathbb{F}\left(  f^{p}\right)  \right]  ^{1/p}.
\]
Now letting $p\longrightarrow\infty$ gives the desired formula.
\end{remark}

\begin{definition}
Let $\left(  E,e,\mathbb{F}\right)  $ be a conditional Riesz triple. Any
element $p\in E^{+}$ satisfying $p\wedge(e-p)=0$ is called a component of $e$.
It's obvious that $p$ is a component of $e$ if and only if $(e-p)$ is so.
\end{definition}

The averaging property fails if $h$ is not positive. We have however the following.

\begin{corollary}
\label{4A}Let $(E,e,\mathbb{F})$ be a conditional Riesz triple. If
$f,h\in\mathcal{L}^{\infty}(\mathbb{F})$ with $h\in R(\mathbb{F})$, then%
\[
M_{\mathbb{F}}(hf)=h^{+}M_{\mathbb{F}}(f)-h^{-}m_{\mathbb{F}}(f).
\]

\end{corollary}

\begin{proof}
Let $p=P_{h^{+}}e.$ Then $h^{+}=ph$ and $-h^{-}=\left(  e-p\right)  h.$ As
$h^{+}\in R\left(  \mathbb{F}\right)  $ we know by \cite[Lemma 3.1]{L-24} that
$p=P_{h^{+}}e\in R\left(  \mathbb{F}\right)  .$ Hence, using Theorem \ref{4F}
we get%
\begin{align*}
M_{\mathbb{F}}(hf)  &  =pM_{\mathbb{F}}(hf)+(e-p)M_{\mathbb{F}}\left(
hf\right)  =M_{\mathbb{F}}(phf)+M_{\mathbb{F}}((e-p)hf)\\
&  =M_{\mathbb{F}}(h^{+}f)+M_{\mathbb{F}}(-h^{-}f)=h^{+}M_{\mathbb{F}%
}(f)-h^{-}m_{\mathbb{F}}(f),
\end{align*}
which proves the corollary.
\end{proof}

It is through conditional expectations that martingales are defined. Similarly
we can use the conditional supremum to define maxingales in a very natural
way. Given a conditional Riesz triple $\left(  E,e,\mathbb{F}\right)  $ and a
filtration $(\mathbb{F}_{n})_{n},,$ we call sub-maxingale every adapted
process $\left(  f_{n}\right)  $ such that $M_{\mathbb{F}_{n}}(f_{n+1})\geq
f_{n}$ for all integer $n.$ The process $\left(  f_{n}\right)  $ is a super
maxingale if $\left(  -f_{n}\right)  $ is a sub-maxingale and it is a
maxingale if it is both sub- and super-maxingale.

\begin{remark}
We can interpret the above definition using the model of portfolio process as
follows: Let $(V_{t})_{t\in\lbrack0,T]}$ be a portfolio process.

\begin{enumerate}
\item Maxingale : the best that can happen in the next step is not better than
the current fortune.

\item Sub-maxingales: we could possibly improve our fortune in the next step.

\item Super-maxingale: we could not possibly increase our fortune in the next step.
\end{enumerate}
\end{remark}

To examine the distinctive features of maxingales, let us treat an example.

\begin{example}
\begin{enumerate}
\item Let $E$ be a Dedekind complete Riesz space and $(\mathbb{F}_{n})_{n}$ a
filtration. If $f\in L^{\infty}(\mathbb{F}_{0})$, then $f_{n}=M_{\mathbb{F}%
_{n}}(f)$ is a maxingale.

\item Let $E$ be a Dedekind complete Riesz space. Let $(\mathbb{F}_{n})_{n}$
be a filtration and $(f_{n})_{n}$ be an adapted process. Consider $g_{n}%
=f_{1}\vee f_{2}\vee...\vee f_{n}$. The sequence $(g_{n})_{n}$ is a
submaxingale . Indeed,%
\begin{align*}
M_{\mathbb{F}_{n}}(g_{n+1})  &  =M_{\mathbb{F}_{n}}(f_{1}\vee f_{2}\vee...\vee
f_{n+1})\\
&  =f_{1}\vee f_{2}\vee...\vee f_{n}\vee M_{\mathbb{F}_{n}}(f_{n+1})\\
&  \geq f_{1}\vee f_{2}\vee...\vee f_{n+1}=g_{n}.
\end{align*}

\end{enumerate}
\end{example}

As a super-maxingale is a super-martingale the following result is an
immediate consequence of \cite[Theorem 3.3]{L-03}.

\begin{proposition}
Let $\left(  E,e,\mathbb{F}\right)  $ be a conditional Riesz triple and
$(f_{i})$ be a super-maxingale. If there exists $g\in E^{+}$ such that
$|f_{i}|\leq g$ for all $i$, then $\left(  f_{i}\right)  $ is order convergent.
\end{proposition}

\section{Some applications for conditional supremum}

In this section we will explore some applications of conditional supremum with
respect to a conditional expectation operator, specifically in the field of
finance. Results regarding ergodic theory will be discussed in a separate
paper. We just say here a few words about this subject. Ergodic theory is
concerned with studying the long term behavior of a dynamic system and
investigating conditions for the equality of space and time means. The
translation of ergodic theory from the framework of probability theory to the
context of Riesz spaces was initiated by Kuo, Labuschagne and Watson in
\cite{L-33}; see also \cite{L-724}. In \cite{L-724} the authors introduced the
conditional expectation preserving system as the 4-uplet $(E,\mathbb{F},S,e)$
where $\left(  E,\mathbb{F},e\right)  $ is a conditional Riesz triple and $S$
is an order continuous Riesz homomorphism on $E$ satisfying $Se=e$ and
$\mathbb{F}Sf=\mathbb{F}f$, for all $f$ in $E$. It was also shown in
\cite{L-724} that the conditional expectation preserving system $(E,\mathbb{F}%
,S,e)$ is ergodic if and only if
\[
\lim_{n\rightarrow\infty}\frac{1}{n}\sum_{k=0}^{n}S^{k}f=\mathbb{F}%
f\;\text{for all invariant}\;f\in E.
\]

Barron, Cardaliaguet and Jensen have proved in \cite{a-1051} a new ergodic
theorem stating that a time maximum is a space maximum. That is a measure
preserving transformation $S$ in a probability space $(\Omega,\mathcal{F}%
,\mathbb{P})$ is ergodic if and only if for each $f\in L^{\infty}(\Omega)$,
\[
\max\limits_{0\leq k<\infty}f(S^{k}\omega)=\operatorname*{essup}%
(f)\;\text{a.s}.
\]

A version of this result in the measure-free framework of Riesz spaces can be
obtained. It sates that the system is ergodic if and only if $M_{\mathbb{F}%
}\left(  f\right)  =\sup\limits_{0\leq k<\infty}S^{k}f$ for all $f\in
L^{\infty}\left(  \mathbb{F}\right)  .$

The following subsection shows that the conditional maximum can be used to
define some useful distance.

\subsection{The operator $\delta$}

In \cite{b-1088} Kaplan introduced an operator called $\delta$ in the
framework of the space of continuous functions on a compact space $X$. This
operator is used later to construct the Riemann integral in \cite{b-1088}. Let
$X$ be a compact Hausdorff space. For $f\in C^{\prime\prime}\left(  X\right)
$ , the bidual of $C\left(  X\right)  ,$ we define $\ell\left(  f\right)  $
and $u\left(  f\right)  $ as follows:%

\[
\ell\left(  f\right)  =\sup\left\{  g\in C\left(  X\right)  :g\leq f\right\}
,\text{ and }u\left(  f\right)  =\inf\left\{  g\in C\left(  X\right)  :f\leq
g\right\}  .
\]
and then $\delta\left(  f\right)  =u\left(  f\right)  -\ell\left(  f\right)
.$ Thus $\delta\left(  f\right)  =0$ if and only if $f\in C\left(  X\right)
.$

We will study this operator in the space $\mathcal{L}^{\infty}(\mathbb{F}).$

\begin{definition}
Let $(E,e,\mathbb{F})$ be a conditional Riesz triple. For each $f\in
\mathcal{L}^{\infty}(\mathbb{F})$ , we define%
\[
\delta(f)=M_{\mathbb{F}}(f)-m_{\mathbb{F}}(f).
\]

\end{definition}

It is immediate that $\delta$ is positive and symmetric, that is
$\delta(f)\geq0$ and $\delta(-f)=\delta(f).$ Furthermore $\delta(f)=0$ if and
only if $f\in R(\mathbb{F})$.

\begin{remark}
The operator $\delta$ can be written differently as follows%
\[
\delta(f)=\inf\{b-a,\;a\leq f\leq b,\;a,b\in R(\mathbb{F})\}.
\]

\end{remark}

The following proposition now subsumes some basic properties of the operator
$\delta$.

\begin{proposition}
The operator $\delta$ is a vector valued semi norm. In particular we have

\begin{enumerate}
\item For every $f$ and $g$ in $\mathcal{L}^{\infty}(\mathbb{F})$,
\[
\delta(f)-\delta(g)\leq\delta(f-g)\leq\delta(f)+\delta(g).
\]

\item For every $f$ in $\mathcal{L}^{\infty}(\mathbb{F})$, $\delta
(f)=\delta(f^{+})+\delta(f^{-}).$
\end{enumerate}
\end{proposition}

\begin{proof}
Our proof starts with the observation that $\delta(\lambda f)=|\lambda
|\delta(f)$ for all $\lambda\in\mathbb{R}$. Even more,%
\[
\delta(f)-\delta(g)=M_{\mathbb{F}}(f)-m_{\mathbb{F}}(f)-M_{\mathbb{F}%
}(g)+m_{\mathbb{F}}(g).
\]
From the inequalities
\[
M_{\mathbb{F}}(f)+m_{\mathbb{F}}(g)\leq M_{\mathbb{F}}(f+g)\leq M_{\mathbb{F}%
}(f)+M_{\mathbb{F}}(g)
\]
and
\[
m_{\mathbb{F}}(f)+M_{\mathbb{F}}(g)\geq m_{\mathbb{F}}(f+g)\geq m_{\mathbb{F}%
}(f)+m_{\mathbb{F}}(g),
\]
we obtain that
\[
\delta(f)-\delta(g)\leq M_{\mathbb{F}}(f+g)-m_{\mathbb{F}}(f+g)\leq
M_{\mathbb{F}}(f)+M_{\mathbb{F}}(g)-m_{\mathbb{F}}(f)-m_{\mathbb{F}}(g)
\]
which gives that
\[
\delta(f-g)\leq\delta(f)+\delta(-g)=\delta(f)+\delta(g).
\]
Combining the last result which the fact that $\delta(-f)=\delta(f)$ gives
(i). Setting $g=0$ in (i) and using $\delta(f+g)=\delta(f\vee g+f\wedge g)$,
we obtain (ii).
\end{proof}

The following result should be obvious.

\begin{proposition}
Let $(E,e,\mathbb{F})$ be a conditional Riesz triple and $f,g$ in
$\mathcal{L}^{\infty}(\mathbb{F})$. If $g\in R(\mathbb{F})$, then
$\delta(f+g)=\delta(f)$.
\end{proposition}

The next proposition will play a crucial role in the proof of our main result.

\begin{proposition}
Let $(E,e,\mathbb{F})$ be a conditional Riesz triple and $f,g$ in
$\mathcal{L}^{\infty}(\mathbb{F})$. Then
\[
||\delta(f)-\delta(g)||_{\infty,\mathbb{F}}\leq2||f-g||_{\infty,\mathbb{F}}.
\]
Thus the operation $\delta$ is norm continuous.
\end{proposition}

\begin{proof}
Let $f,g$ in $\mathcal{L}^{\infty}(\mathbb{F})$ be fixed. Then we have%
\begin{align*}
||\delta(f)-\delta(g)|| &  =||M_{\mathbb{F}}(f)-m_{\mathbb{F}}%
(f)-(M_{\mathbb{F}}(g))-m_{\mathbb{F}}(g)||\\
&  \leq||M_{\mathbb{F}}(f)-M_{\mathbb{F}}(g)||+||m_{\mathbb{F}}%
(f)-m_{\mathbb{F}}(g)||\\
&  =M_{\mathbb{F}}(|M_{\mathbb{F}}(f)-M_{\mathbb{F}}(g)|)+M_{\mathbb{F}%
}(|m_{\mathbb{F}}(f)-m_{\mathbb{F}}(g)|).
\end{align*}
Proposition \ref{4B} gives%
\[
||\delta(f)-\delta(g)||\leq M_{\mathbb{F}}(|f-g|)+M_{\mathbb{F}}%
(|f-g|)=2||f-g||.
\]
This completes the proof.
\end{proof}

The following theorem gives the expression of the distance from an element $f$
in $\mathcal{L}^{\infty}(\mathbb{F})$ to the range of $\mathbb{F}$. It has
been already obtained for the space of continuous functions on a compact $X$
by Kaplan in \cite{kaplan2011bidual}.

\begin{theorem}
Let $(E,e,\mathbb{F})$ be a conditional Riesz triple. For all $f$ in
$\mathcal{L}^{\infty}(\mathbb{F})$, the distance of $f$ from $R(\mathbb{F})$
is
\[
d(f,R(\mathbb{F})):=\inf\left\{  \left\Vert f-g\right\Vert _{\mathbb{F}%
,\infty}:g\in R\left(  \mathbb{F}\right)  \right\}  =\frac{1}{2}\delta(f),
\]
and this distance is attained in $g=\dfrac{m_{\mathbb{F}}\left(  f\right)
+M_{\mathbb{F}}\left(  f\right)  }{2}$.
\end{theorem}

\begin{proof}
We shorten $M_{\mathbb{F}}\left(  f\right)  =M$ and $m_{\mathbb{F}}\left(
f\right)  =m.$ For all $g$ in $R(\mathbb{F})$, we have $\delta(g)=0$. Hence by
the above proposition, we obtain
\[
\delta\left(  f\right)  =\left\Vert \delta(f)\right\Vert =\left\Vert
\delta\left(  f\right)  -\delta\left(  g\right)  \right\Vert \leq||f-g||
\]
which implies that%
\[
\frac{1}{2}\delta(f)=\frac{1}{2}||\delta(f)||\leq d(f,R(\mathbb{F})).
\]
Now for $g=\dfrac{1}{2}(M+m)$ we have%
\begin{align*}
\left\vert f-g\right\vert  &  =\left\vert f-\frac{1}{2}(M+m)\right\vert \\
&  =\left(  f-\frac{1}{2}(M+m)\right)  \vee\left(  \frac{1}{2}(M+m)-f\right)
\\
&  \leq\left(  M-\frac{1}{2}(M+m)\right)  \vee\left(  \frac{1}{2}%
(M+m)-m\right) \\
&  =\frac{1}{2}(M-m)=\frac{1}{2}\delta(f).
\end{align*}
Now we apply $M_{\mathbb{F}}$ to get the inequality%
\[
\left\Vert f-g\right\Vert \leq\dfrac{1}{2}\delta\left(  f\right)  .
\]
This completes the proof.
\end{proof}

The notion of conditional supremum can also be applied in finance as it will
be shown in our last subsection.

\subsection{Financial markets}

We consider a filtration $(\mathbb{F}_{t})_{t\in\left[  |0,T|\right]  }$ in a
Dedekind complete Riesz space $E$, where $T$ is the time horizon which
represents the expiration date, that is, the length of time the trader intends
to hold the particular investment. We consider an adapted, vector-valued,
non-negative process $(S_{t})_{t\in\left[  0,\mathbb{F}\right]  }$, where for
$t\in\left[  0,T\right]  ,$ $S_{t}\in R\left(  \mathbb{F}_{t}\right)  $ and
represents the price of some risky asset in the financial market in
consideration. Trading strategies are given by a adapted process $(\theta
_{t})_{t\in\left[  0,T\right]  }$ where for $t\in\left[  0,T\right]  $,
$\theta_{t}$ represents the quantities invested in the risky asset at time
$t$. The liquidation value at time $t$ is given by $V_{t}=\theta_{t}S_{t}.$ We
assume throughout that trading is self financing, that is, it satisfies the
following conditions%
\[
\theta_{t+1}S_{t}=\theta_{t}S_{t},\qquad1\leq t\leq T-1.
\]
It follows that
\[
\Delta V_{t+1}:=V_{t+1}-V_{t}=\theta_{t}\left(  S_{t+1}-S_{t}\right)
=\theta_{t}\Delta S_{t+1}.
\]
The liquidation value at time $T$ of a portfolio $\theta$ starting from
initial capital $p_{t_{0}}$ at time $t_{0}$ is then given by the formula%
\[
v_{t_{0},T}=p_{t_{0}}+\sum\limits_{i=t_{0}}^{T}\theta_{i-1}\Delta S_{i}.
\]
For $t\in\left[  0,T\right]  $ we denote by $\mathcal{V}_{t,T}$ the set of all
$T$ terminal elementary portfolios, starting with zero initial capital at time
$t$ :%
\[
\mathcal{V}_{t,T}=\left\{  v_{t,T}=\sum\limits_{i=t}^{T}\theta_{i-1}\Delta
S_{i},\theta_{i}\in R(\mathbb{F}_{i})\right\}  .
\]

\begin{definition}
\label{pT-1}An amount of money (or, contingent claim in the financial
vocabulary) $h_{T}\in R(\mathbb{F}_{T})$ is said to be super-hedgeable at time
$t$ if there exists $p_{t}\in R(\mathbb{F}_{t})$ such that starting with this
initial capital, one can find a portfolio strategy $v_{t,T}\in\mathcal{V}%
_{t,T}$ allowing him to exceed $h_{T}$ at time $T,$ which can be expressed as
follows:%
\[
p_{t}+v_{t,T}\geq h_{T}.
\]
We say that $p_{t}$ is a price for $h_{t}.$
\end{definition}

The set of all super-hedgeable claims with zero initial endowment at time $t$
is then given by%
\[
A_{t,T}=\left\{  v_{t,T}-\varepsilon_{T},v_{t,T}\in V_{t,T},\varepsilon_{T}\in
R(\mathbb{F}_{T})^{+}\right\}  .
\]

Let $P_{t,T}(h_{T})$ be the set of all super-hedging prices $p_{t}\in
R(\mathbb{F}_{t})$ at time $t$ for the contingent claim $h_{T}$ as in
Definition \ref{pT-1} and the minimal super-hedging price is
\[
\pi_{t,T}(h_{T}):=m_{\mathbb{F}_{t}}(P_{t}(h_{T})).
\]

We denote by $P_{t,T}=P_{t,T}(0)$ the set of all super-hedging prices for the
zero claim, precisely%

\[
P_{t,T}=\left\{  p_{t}\in R(\mathbb{F}_{t}):\text{ there exists }v_{t,T}%
\in\mathcal{V}_{t,T}:p_{t}+v_{t,T}\geqslant0\right\}  .
\]

We start by a simple lemma which follows immediately from the definitions
given above.

\begin{lemma}
\label{L1}We have $P_{t,T}=\left(  -A_{t,T}\right)  \cap R(\mathbb{F}_{t})$
and%
\[
P_{t,T}=\left\{  M_{\mathbb{F}_{t}}(-v_{t,T}):v_{t,T}\in\mathcal{V}%
_{t,T}\right\}  +R(\mathbb{F}_{t})^{+}.
\]

\end{lemma}

We will say that an immediate profit holds at time $t<T$ if it is possible to
exceed the zero contingent claim from a negative initial price at time $t$,
i.e. $P_{t,T}\cap R(\mathbb{F}_{t})^{-}$ is not reduced to $0$. On the
contrary, we say that the Absence of Immediate Profit (AIP) holds at time $t.$
In the sequel we will be interested in financial markets satisfying AIP
conditions, that is,%
\begin{equation}
P_{t,T}\cap R(\mathbb{F}_{t})^{-}=A_{t,T}\cap R(\mathbb{F}_{t})^{+}%
=\{0\},\text{ for all }t<T. \label{NGD}%
\end{equation}

The following result provides a simple, but very useful, characterization of
AIP condition.

\begin{lemma}
The following statements are equivalent for a financial market.

\begin{enumerate}
\item[(i)] The AIP condition holds;

\item[(ii)] For any $t\leq T$ and for any $v_{t,T}\in\mathcal{V}%
_{t,\mathbb{F}}$, $m_{\mathbb{F}_{t}}(v_{t,\mathbb{F}})\leqslant0.$
\end{enumerate}
\end{lemma}

Condition (ii) can be financially interpreted as a possibility of loss.

\begin{proof}
This follows from Definition \ref{pT-1} and Lemma \ref{L1}. Indeed we have%
\begin{align*}
\text{AIP}  &  \Leftrightarrow P_{t,\mathbb{F}}\subset R(\mathbb{F}_{t}%
)^{+}\Leftrightarrow M_{\mathbb{F}_{t}}(-v_{t,T})\geqslant0\;\forall
\;v_{t,T}\in\mathcal{V}_{t,\mathbb{F}}\\
&  \Leftrightarrow m_{\mathbb{F}_{t}}(v_{t,\mathbb{F}})\leqslant0,\;\forall
v_{t,\mathbb{F}}\in\mathcal{V}_{t,\mathbb{F}}.
\end{align*}

\end{proof}

The following theorem provides a characterization of the AIP, which is a
generalization of \cite[Proposition 3.4]{a-1291}.

\begin{theorem}
The AIP condition holds if and only if, for all $0\leq t<T$, $m_{\mathbb{F}%
_{t}}(S_{t+1})\leq S_{t}\leq M_{\mathbb{F}_{t}}(S_{t+1})$.
\end{theorem}

A process satisfying the condition%
\[
M_{\mathbb{F}_{t}}(S_{t+1})\geq S_{t}.
\]
in the above theorem is called submaxingale.

\begin{proof}
Suppose that AIP holds. Consider $v_{t,\mathbb{F}}=\Delta S_{t}$. By the lemma
above, we have $m_{\mathbb{F}_{t}}(v_{t,T})\leqslant0$, that is,%
\[
m_{\mathbb{F}_{t}}(S_{t+1})\leqslant S_{t}.
\]
Similarly, choosing $v_{t,T}=-\Delta S_{t}$, we get that%
\[
M_{t}(S_{t+1})\geqslant S_{t}.
\]
Conversely it follows from our assumptions that%
\begin{equation}
M_{\mathbb{F}_{t-1}}(\Delta S_{t})\geq0\text{ and }m_{\mathbb{F}_{t-1}}(\Delta
S_{t})\leq0.\label{4E1}%
\end{equation}
Let $v_{t,T}\in\mathcal{V}_{t,T}.$ We have to show that $m_{\mathbb{F}_{t}%
}(\ v_{t,\mathbb{F}})\leqslant0$. Write%
\[
v_{t,T}=\sum\limits_{t=1}^{T}\theta_{t-1}\Delta S_{t},
\]
with $\theta_{t}\in R(\mathbb{F}_{t}).$ By Corollary \ref{4A} and (\ref{4E1})
we have for all $t\geq1,$%
\[
M_{\mathbb{F}_{t-1}}(\theta_{t-1}\Delta S_{t})=\theta_{t-1}^{+}M_{\mathbb{F}%
_{t-1}}(\Delta S_{t})-\theta_{t-1}^{-}m_{\mathbb{F}_{t-1}}(\Delta S_{t})\geq0.
\]
Consequently,%
\[
m_{\mathbb{F}_{t-1}}(\theta_{t-1}\Delta S_{t})=-M_{\mathbb{F}_{t-1}}%
(-\theta_{t-1}\Delta S_{t})\leq0.
\]
Using the fact that $m_{\mathbb{F}}$ is $\mathbb{F}$-translation invariant
(See Proposition \ref{4B}.(ii)) we get%
\[
m_{\mathbb{F}_{T-1}}(v_{t,T})=\sum\limits_{t=1}^{T-1}\theta_{t-1}\Delta
S_{t}+m_{\mathbb{F}_{T-1}}(\theta_{T-1}\Delta S_{T})\leq\sum\limits_{t=1}%
^{T-1}\theta_{t-1}\Delta S_{t}.
\]
Applying successively $m_{\mathbb{F}_{T-2}},m_{\mathbb{F}_{T-3}}%
,...,m_{\mathbb{F}t}$ we get at the end $m_{\mathbb{F}_{t}}\left(
v_{t,\mathbb{F}}\right)  \leq0$.
\end{proof}

\begin{corollary}
The AIP condition holds if and only if%
\[
m_{\mathbb{F}_{t}}(S_{u})\leq S_{t}\leq M_{\mathcal{\mathbb{F}}_{t}}(S_{u}),
\]
for all $0\leq t\leq u\leq T.$
\end{corollary}

Now we introduce the no arbitrage condition. An arbitrage opportunity is the
possibility to super-replicate the zero contingent claim from a zero initial
endowment. More precisely:

\begin{definition}
An arbitrage opportunity at time $t<T$ is a terminal continuous-time portfolio
process $v_{t,T}\in\mathcal{V}_{t,T}$ starting from the zero initial endowment
at time $t$ such that $v_{t,T}\geq0$. If there is no arbitrage opportunity, we
say that NA condition holds, i.e. $\mathcal{V}_{t,T}\cap R(\mathbb{F}_{T}%
)^{+}=\{0\}$ for any $t\leq T$. Which is equivalent to $A_{t,T}\cap
R(\mathbb{F}_{T})^{+}=\{0\}$.
\end{definition}

\begin{remark}
It is clear that NA implies AIP.
\end{remark}

\begin{proposition}
If for every component $p>0$ we have $p(S_{t}-m_{\mathcal{\mathbb{F}}_{t}%
}(S_{t+1}))>0$ and $p(M_{\mathcal{\mathbb{F}}_{t}}(S_{t+1})-S_{t})>0$ for any
$t\leq T-1$, then NA is equivalent to AIP.
\end{proposition}

\begin{proof}
We already know that NA implies AIP. Assume that AIP holds. Let us show that,
for any $t\leq T,$%
\[
R_{t}^{T}\cap R\left(  \mathbb{F}_{t}\right)  ^{+}=\{0\}.
\]
Consider first the one step model, i.e. suppose that $\gamma_{T}=\theta
_{T-1}\Delta S_{T}-\varepsilon_{T}\in R_{T-1}^{T}$ is such that $\gamma
_{T}\geq0$. So, $\theta_{T-1}\Delta S_{T}\geq0$. Suppose that $\theta
_{T-1}\neq0$. If $\theta_{T-1}^{+}\neq0$ we have $\theta_{T-1}^{+}\theta
_{T-1}\Delta S_{T}\geq0$ so $(\theta_{T-1}^{+})^{2}\Delta S_{T}\geq0$. Then
\[
((\theta_{T-1}^{+})^{2}\vee e)((\theta_{T-1}^{+})^{2}\wedge e)\Delta S_{T}%
\geq0.
\]
As $((\theta_{T-1}^{+})^{2}\vee e)$ is invertible, if we consider
$p=((\theta_{T-1}^{+})^{2}\wedge e)$ we obtain $p\Delta S_{T}\geq0$. Then
$p(S_{T-1}-m_{_{\mathbb{F}_{1}}}(S_{T}))\leq0$. Note that $p$ is a strictly
positive component which leads to a contradiction. We deduce that
$\theta_{T-1}^{+}=0$ . We obtain $\theta_{T-1}=-\theta_{T-1}^{-}$ , so
$-\theta_{T-1}^{-}\Delta S_{T}\geq0$. Similarly we show that $\theta_{T-1}%
^{-}$. We deduce that $\theta_{T}=0$ and $\gamma_{T}=-\varepsilon_{T}\leq0$.
As we suppose that $\gamma_{T}\geq0$, $\gamma_{T}=0$.

Suppose that $R_{t+1}^{T}\cap R\left(  \mathbb{F}_{T}\right)  ^{+}=\{0\}$ by
induction. Consider a terminal claim $\gamma_{T}=\sum\limits_{t\leq u\leq
T-1}\theta_{u}\Delta S_{u+1}-\varepsilon_{T}\in R_{t}^{T}$ such that
$\gamma_{T}\geq0$. We have%
\[
\theta_{t}\Delta S_{t+1}+\sum\limits_{t+1\leq u\leq T-1}\theta_{u}\Delta
S_{u+1}=\gamma_{T}+\varepsilon_{T}\geq0.
\]
Therefore, $\theta_{t}\Delta S_{t+1}\in\mathcal{P}_{t+1,T}$. Under AIP, we
deduce that $\theta_{t}\Delta S_{t+1}\geq0$. By the first step, we get that
$\theta_{t}=0$. Therefore,
\[
\gamma_{T}\in R_{t+1}^{T}\cap R\left(  \mathbb{F}_{T}\right)  ^{+}=\{0\}
\]
by the induction hypothesis.
\end{proof}

\end{document}